\documentclass{amsart}
\usepackage{amsthm, amsmath,amsfonts, amssymb, enumerate}
\usepackage{graphicx,tikz, color, MnSymbol}
\usepackage[latin2]{inputenc}

\newcommand{\angled}[1]{\langle#1\rangle}

\newcommand{\Z}{\mathbb Z}
\newcommand{\R}{\mathbb R}

\newcommand{\G}{\mathcal G}

\newcommand{\p}{\mathcal P}

\newcommand{\Ga}{\Gamma}

\newcommand{\h}{\mathcal H}

\newcommand{\acts}{\curvearrowright}

\newtheorem{thm}{Theorem}[section]
\newtheorem{lem}[thm]{Lemma}

\newtheorem{cor}[thm]{Corollary}

\theoremstyle{definition}
\newtheorem{defn}[thm]{Definition}

\title{Cocompactly cubulated 2-dimensional Artin groups}
\author{Jingyin Huang}
\author{Kasia Jankiewicz}
 \author{Piotr Przytycki}
\address{Dept. of Math. \& Stats.\\
                    McGill University \\
                    Montreal, Quebec, Canada H3A 0B9}
           \email{jingyin.huang@mcgill.ca}
           \email{kasia@math.mcgill.ca}
           \email{piotr.przytycki@mcgill.ca}

\begin{document}

\maketitle
\begin{abstract}
\noindent We give a necessary and sufficient condition for a
2-dimensional or a three-generator Artin group $A$ to be
(virtually) cocompactly cubulated, in terms of the defining graph
of $A$.

\end{abstract}

\section{Introduction}

We say that a group is (\emph{cocompactly}) \emph{cubulated} if
it acts properly (and compactly) by combinatorial automorphisms on a CAT(0) cube complex. We
say that a group is \emph{virtually cocompactly cubulated}, if it has a
finite index subgroup that is cocompactly cubulated.
Such groups fail to
have Kazhdan's property (T) \cite{NR}, are bi-automatic \cite{Swiat}, satisfy the Tits
Alternative \cite{SW} and, if cocompactly cubulated, they satisfy rank-rigidity \cite{caprace2011rank}. For more background on CAT(0) cube complexes, see
the survey article of Sageev \cite{survey}.

The \emph{Artin group} with generators $s_i$ and
exponents $m_{ij}=m_{ji}\geq 2$, where $i\neq~j$, is presented by
relations
$\underbrace{s_is_js_i\cdots}_{m_{ij}}=\underbrace{s_js_is_j\cdots}_{m_{ij}}$.
Its \emph{defining graph} has vertices
corresponding to $s_i$ and edges labeled $m_{ij}$ between
$s_i$ and $s_j$ whenever $m_{ij}<\infty$.

Artin groups that are \emph{right-angled} (i.e.\ the ones with
$m_{ij}\in \{2,\infty\}$) are cocompactly cubulated,
and they play a prominent role in theory of special
cube complexes of Haglund and Wise. However, much less is known
about other Artin groups, in particular about braid groups. In
\cite{Hier} Wise suggested an approach to cubulating Artin
groups using cubical small cancellation. However, we failed to execute this approach: we were not able to establish the B(6) condition.

In this article we consider Artin groups that have three
generators, or are $2$-\emph{dimensional}, that is, their
corresponding Coxeter groups have finite special subgroups of
maximal rank $2$ (or, equivalently, $2$-dimensional Davis
complex). We characterise when such a group is virtually
cocompactly cubulated. This happens only for very rare defining
graphs. An \emph{interior} edge of a graph is an edge that is
not a leaf.

\begin{thm}
\label{thm:main}
Let $A$ be a $2$-dimensional Artin group. Then the following
are equivalent.
\begin{enumerate}[(i)]
\item $A$ is cocompactly cubulated,
\item $A$ is virtually cocompactly cubulated,
\item each connected component of the defining graph of $A$
    is either
\begin{itemize}
\item
    a vertex, or an edge, or else
\item
   all its interior edges are labeled by $2$ and all its leaves are
    labelled by even numbers.
\end{itemize}
\end{enumerate}

Moreover, if $A$ is an arbitrary Artin group, then (iii)
implies (i).
\end{thm}

\begin{thm}
\label{thm:3gen}
Let $A$ be a three-generator Artin group. Then the following
are equivalent.
\begin{enumerate}[(i)]
\item $A$ is cocompactly cubulated,
\item $A$ is virtually cocompactly cubulated,
\item the defining graph of $A$ is as in
    Theorem~\ref{thm:main}(iii) or has two edges labelled
    by~2.
\end{enumerate}
\end{thm}

\subsection{Remarks}

From Theorem~\ref{thm:3gen} it follows that the $4$-strand
braid group is not virtually cocompactly cubulated.

Note that, independently, Thomas Haettel has obtained a full
classification of cocompactly cubulated Artin groups. We intend
to bring with Haettel our results to common denominator and
prove that an Artin group is virtually cocompactly cubulated
only if it is cocompactly cubulated.

The equivalence of (i) and (ii) has no
counterpart for Coxeter groups, where the group
$\widetilde{A}_2$ generated by reflections in the sides of an
equilateral triangle in $\R^2$ is virtually cocompactly
cubulated, but not cocompactly cubulated.

There are Artin groups that do not satisfy the equivalent
conditions from Theorem~\ref{thm:main}, but are cubulated.
Namely, it follows from
\cite{brunner1992geometric,hermiller1999artin} that if the
defining graph of $A$ is a tree, then $A$ is the fundamental
group of a link complement that is a graph manifold with
boundary. Hence by the work of Liu \cite{Liu} or Przytycki and
Wise \cite{PW1} the Artin group $A$ is cubulated.

Artin groups of \emph{large type}, that is, with all
$m_{ij}\geq 3$ are $2$-dimensional. For many of them Brady and
McCammond constructed 2-dimensional CAT(0) complexes with
proper and cocompact action \cite{BradyMccammond2000}. However,
these complexes are built of triangles, not squares.

\subsection{Some historical background}

Sageev invented a way of \emph{cubulating} groups (i.e.\
showing that they are cubulated) using codimension 1-subgroups
\cite{Sag}, which was later also explained in the language of
\emph{walls} in the Cayley complex of the group \cite{CN,Nica}.
Here we give a brief account on some cubulation results, for a
more complete one see \cite{HruW}.

Using the technology of walls, Niblo and Reeves cubulated Coxeter
groups \cite{NR} and Caprace and M\"uhlherr analysed when
this cubulation is cocompact \cite{CapraceM}. It is not known if all
Coxeter groups are virtually cocompactly cubulated. Wise cocompactly cubulated small cancellation groups \cite{Wsmall},
and Ollivier and Wise cocompactly cubulated random groups at
density $< \frac{1}{6}$ \cite{OW}.

Furthermore, using the surfaces of Kahn and Markovic, Bergeron
and Wise cocompactly cubulated the fundamental groups of closed
hyperbolic 3-manifolds \cite{KM,BW}, and later Wise cocompactly
cubulated the fundamental groups of compact hyperbolic
3-manifolds with boundary \cite{Hier}. Hagen and Wise
cocompactly cubulated hyperbolic free-by-cyclic groups
\cite{HagW}.

Groups that are not (relatively) hyperbolic are harder to
cubulate cocompactly. Przytycki and Wise cubulated the
fundamental groups of all compact 3-dimensional manifolds that
are not graph manifolds, as well as graph manifolds with
boundary \cite{PW1,PW2}. In \cite{Liu} Liu gave a criterion for a graph
manifold fundamental group to be virtually cubulated \emph{specially}
(meaning that the quotient of the action admits a local isometry into the
Salvetti complex of a right-angled Artin group), but we do not
know if this is equivalent to just being cubulated. Hagen and Przytycki gave a criterion for a
graph manifold fundamental group to be cocompactly cubulated
\cite{hagen-przytycki}. In general, it is difficult to find obstructions for
groups to be cubulated. Another result of this type is Wise's
characterization of tubular groups that are cocompactly
cubulated \cite{Wise:Tubular}.

\subsection{Proof outline for (i)$\Rightarrow$(iii) in Theorem~\ref{thm:main}}
Given a $2$-dimensional Artin group acting properly and
cocompactly on a CAT(0) cube complex, we show that its two-generator special subgroups are convex cocompact. More
precisely, each of them acts cocompactly on a convex subcomplex which
naturally decomposes as a product of a vertical factor and a
horizontal factor. Geometrically, the intersection of two such
subgroups is either vertical or horizontal. However, if Theorem~\ref{thm:main}(iii) is not satisfied, then
this intersection is neither vertical nor horizontal by algebraic
considerations.

One of the ingredients of the proof is Theorem~\ref{product of
hyperbolic}, which asserts that a top rank product of hyperbolic
groups acting on a CAT(0) cube complex is always convex
cocompact.

\subsection{Organization}
In Section~\ref{sec:preliminaries} we give some background on
CAT(0) spaces and CAT(0) cube complexes. Section~\ref{sec:cube
complexes} is devoted to the proof of Theorem~\ref{product of
hyperbolic}. In Section~\ref{sec:artin groups} we give some
background on Artin groups and discuss some algebraic
properties of two-generator Artin groups. Finally, in
Section~\ref{sec:main} we prove Theorem~\ref{thm:main} and in
Section~\ref{sec:three} we prove Theorem~\ref{thm:3gen}.
\subsection{Acknowledgements}
The authors would like to thank Daniel T. Wise for helpful discussions. The third author was partially supported by National Science Centre DEC-2012/06/A/ST1/00259 and NSERC.

\section{Preliminaries}\label{sec:preliminaries}
A group is a \emph{$\mathrm{CAT(0)}$ group} if it acts properly
and cocompactly on a CAT(0) space. We assume the reader is
familiar with the basics of CAT(0) spaces and groups. For
background, see~\cite{MR1744486}. In this section we collect
some less classical results.

\subsection{Asymptotic rank}
The following definition was introduced in
\cite{kleiner1999local}.

\begin{defn}
Let $X$ be a $\mathrm{CAT}(\kappa)$ space. For $x\in X$ we
denote by $\Sigma_{x}X$ the CAT(1) space that is the completion
of the space of directions at $x$ \cite[Definition
II.3.18]{MR1744486}. The \textit{geometric dimension} of $X$,
denoted $\mathrm{GeomDim}(X)$ is defined inductively as
follows.
\begin{itemize}
\item $\mathrm{GeomDim}(X)=0$ if $X$ is discrete,
\item $\mathrm{GeomDim}(X)\le n$ if
    $\mathrm{GeomDim}(\Sigma_{x}X)\le n-1$ for any $x\in
    X$.
\end{itemize}
\end{defn}

\begin{defn}
Let $X$ be a CAT(0) space. Then its \textit{asymptotic rank},
denoted by $\mathrm{asrk}(X)$, is the supremum of the geometric
dimension of the asymptotic cones of~$X$.
\end{defn}


\begin{thm}
\label{thm:asrk}
Let $X$ and $Y$ be $\mathrm{CAT(0)}$ spaces.
Then
\begin{enumerate}
\item $\mathrm{asrk}(X\times Y)\ge
    \mathrm{asrk}(X)+\mathrm{asrk}(Y)$,
\item if $\mathrm{asrk}(X)\le 1$, then $X$ is hyperbolic.
\end{enumerate}
\end{thm}

The first assertion follows from Theorem A of
\cite{kleiner1999local} and the second assertion follows
from Corollary 1.3 of \cite{wenger2007asymptotic}.

\begin{defn}
If $G$ is a CAT(0)
group acting properly and cocompactly on a CAT(0) space $X$, then the
\textit{asymptotic rank} of $G$ is the asymptotic rank of $X$.
By \cite[Theorem~C]{kleiner1999local} this is the maximal $n$
for which there is a quasi-isometric embedding $\R^n\to X$. Hence it does not depend on the choice
of the CAT(0) space~$X$.
\end{defn}

\begin{lem}
\label{rank} Suppose that $G$ is a $\mathrm{CAT(0)}$ group, and
that $G$ acts properly and cocompactly on a contractible
$n$-dimensional cell complex $X$ (not necessarily $\mathrm{CAT(0)}$). Then the asymptotic rank of
$G$ is~$\le n$.
\end{lem}

\begin{proof}
Choose any $G$-equivariant length metric on $X$. We will prove that
there does not exist a quasi-isometric embedding $f:\R^k\to X$
for $k>n$. Otherwise, since $X$ is contractible and admits a
cocompact action of~$G$, we can assume that $f$ is a continuous
quasi-isometry: such $f$ can be defined by induction on
consecutive skeleta of the standard cubical subdivision of
$\R^k$.

Let $Y\subseteq X$ be the smallest subcomplex containing
$f(\R^k)$. Then $f:\R^k\to Y$ is a quasi-isometry. Let
$g:Y\to\R^k$ be a quasi-isometry inverse to $f$, we can again
assume that $g$ is continuous. For any $x\in\R^k$ the distance
$d(g\circ f(x),x)$ is uniformly bounded and consequently there
is a proper geodesic homotopy between $g\circ f$ and the
identity map.

Recall that for a topological space $X$ we can consider
\emph{locally finite chains} in~$X$, which are formal sums
$\Sigma_{\lambda\in\Lambda}a_{\lambda}\sigma_{\lambda}$ where
$a_{\lambda}$ are integers, $\sigma_{\lambda}$ are singular
simplices, and any compact set in $X$ intersects the images of only
finitely many $\sigma_{\lambda}$ with $a_{\lambda}\neq 0$. This
gives rise to \textit{locally finite homology} of $X$, denoted
by $H^{\mathrm{lf}}_{*}(X)$. Moreover, proper maps induce
homomorphisms on locally finite homology. See \cite[Section
2.2]{bestvina2008quasiflats} for more discussion.

Since there is a proper geodesic homotopy between $g\circ f$ and
the identity map, $g\circ f$ induces the identity on
$H^{\mathrm{lf}}_{*}(\R^k)$, and consequently
$f_{*}\colon H^{\mathrm{lf}}_k(\R^k)\to
H^{\mathrm{lf}}_k(Y)$ is injective. This leads to a
contradiction, since $H^{\mathrm{lf}}_k(\R^k)$ contains the
fundamental class $[\R^k]$ which is a nontrivial element, while
$H^{\mathrm{lf}}_k(Y)=0$ since $\dim(Y)<k$.
\end{proof}

\subsection{Gate and parallel set}
All CAT(0) cube complexes in our article are
finite-dimensional. Throughout this paper the only metric that
we consider on a CAT(0) cube complex $X$ is the CAT(0) metric
$d$. The \emph{convex hull} of a subspace $Y\subseteq X$ is the
smallest convex subspace containing $Y$, and is not necessarily
a subcomplex, while the \emph{combinatorial convex hull} of $Y$
is the smallest convex subcomplex of $X$ containing $Y$. For a
complete convex subspace $Y\subseteq X$ we denote by $\pi_Y\colon X\to
Y$ the closest point projection onto $Y$.

The following lemma was proved in slightly different contexts
by various authors
\cite{behrstock2014hierarchically,huang2014top,MR2421136,abramenko2008buildings}:

\begin{lem}\cite[Lemma 2.10]{huang2014top}
\label{gate} Let $X$ be a $\mathrm{CAT(0)}$ cube complex of
dimension~$n$, and let $Y_{1}$, $Y_{2}$ be convex subcomplexes.
Let $\Delta=d(Y_{1},Y_{2})$, $V_{1}=\{y\in Y_{1}\mid
d(y,Y_{2})=\Delta\}$ and $V_{2}=\{y\in Y_{2}\mid
d(y,Y_{1})=\Delta\}$. Then:
\begin{enumerate}
\item $V_{1}$ and $V_{2}$ are nonempty convex
    subcomplexes.
\item $\pi_{Y_{1}}$ maps $V_{2}$ isometrically onto $V_{1}$
    and $\pi_{Y_{2}}$ maps $V_{1}$ isometrically onto
    $V_{2}$. Moreover, the convex hull of $V_{1}\cup V_{2}$
    is isometric to $V_{1}\times [0,\Delta]$.
\item for every $\epsilon>0$ there exists
    $\delta=\delta(\Delta,n,\epsilon)$ such that if
    $y_{1}\in Y_{1}$, $y_{2}\in Y_{2}$ and
    $d(y_{1},V_{1})\geq \epsilon$, $d(y_{2},V_{2})\geq
    \epsilon$, then
\begin{equation*}
\label{2.11}
d(y_{1}, Y_{2})\geq \Delta + \delta d(y_{1},V_{1})\,,\ d(y_{2}, Y_{1})\geq \Delta + \delta d(y_{2},V_{2})\,.
\end{equation*}
\end{enumerate}
\end{lem}

We call $V_{1}\subseteq Y_1$ the \textit{gate with respect to~$Y_{2}$}, and $V_{2}\subseteq Y_2$ the \textit{gate with
respect to~$Y_{1}$}. We write $\G(Y_{1},Y_2)=(V_1,V_2)$. We say
that $Y_1,Y_2$ are \emph{parallel} if
$\G(Y_{1},Y_2)=(Y_1,Y_2)$.

\begin{lem}[{\cite[Lemma~2.9]{huang2014quasi}}]
\label{combinatorial description} Let $X$ be a
$\mathrm{CAT(0)}$ cube complex, and let $(V_{1}, V_{2}) =
\G(Y_{1},Y_{2})$ for some convex subcomplexes $Y_1,Y_2\subseteq
X$. Let $e$ be an edge in $V_{1}$ and let $h$ be the hyperplane
dual to $e$. Then $h\cap V_{2}\neq\emptyset$.
\end{lem}

\begin{lem}[{\cite[Lemma~2.5]{caprace2011rank}}]
\label{product decomposition cube} A decomposition of a
$\mathrm{CAT(0)}$ cube complex as a product of
$\mathrm{CAT(0)}$ cube complexes corresponds to a partition
$\h_1\sqcup\h_2$ of the collection of hyperplanes of $X$ such
that every hyperplane in $\h_1$ intersects every hyperplane
in~$\h_2$.
\end{lem}

\begin{lem}
\label{parallel set} Let $X$ be a $\mathrm{CAT(0)}$ cube
complex and let $Y\subseteq X$ be a convex subcomplex. Let
$\{Y_{\lambda}\}_{\lambda\in\Lambda}$ be the collection of all
convex subcomplexes that are parallel to $Y$. Then the
combinatorial convex hull $P_Y$ of
$\bigcup_{\lambda\in\Lambda}Y_{\lambda}$ admits a natural
product decomposition $P_Y=Y\times Y^{\perp}$.
\end{lem}

$P_Y$ is called the \emph{combinatorial parallel set} of $Y$.

\begin{proof}
Let $\h$ be the collection of hyperplanes in $X$ that separate
some points in $\bigcup_{\lambda\in\Lambda}Y_{\lambda}$ and let
$h\in\h$. We claim that either $h$ intersects all $Y_{\lambda}$ or it is disjoint from all $Y_{\lambda}$. Indeed, we have
$\G(Y,Y_{\lambda})=(Y,Y_{\lambda})$ for all
$\lambda\in\Lambda$. It follows from Lemma \ref{combinatorial
description} that if $h$ intersects some $Y_{\lambda}$, then it intersects $Y$, and
hence it intersects all $Y_{\lambda}$.

Let $\h_1$ and $\h_2$ be the collections of hyperplanes
satisfying the first assertion and the second assertion in the
claim, respectively. For any $h\in\h_2$, there exist
$\lambda,\lambda'\in\Lambda$ such that $h$ separates
$Y_{\lambda}$ from $Y_{\lambda'}$. Thus $h$ intersects every
hyperplane in $\h_1$. Note that $\h$ is the collection of
hyperplanes that intersect $P_Y$ and $\h_1$ is the collection
of hyperplanes that intersect $Y$. Thus by Lemma \ref{product
decomposition cube}, $P_{Y}$ admits a product decomposition
$P_Y=Y\times Y^{\perp}$.
\end{proof}

\section{Cocompact cores}\label{sec:cube complexes} The main goal of this section is to
prove Theorem~\ref{product of hyperbolic} on existence of
cocompact cores for top rank products of hyperbolic groups. The
first step towards it is to study flats in a CAT(0) cube
complex, which we do in Section~\ref{hyperplane intersect
flat}. A hurried reader can proceed directly to
Section~\ref{sec:product} and use \cite[Theorem
2.6]{WiseWoodhouse} instead. However, our Theorem~\ref{flat} is of independent interest.

\subsection{Combinatorial convex hull of a flat}
Throughout this paper a \emph{flat} is a CAT(0) flat,
i.e.\ an isometrically embedded copy of $\R^n$, not
necessarily combinatorial. A \emph{half-flat} is an
isometrically embedded copy of $\R^{n-1}\times [0,\infty)$.

\begin{lem}
\label{hyperplane intersect flat} Let $X$ be a
$\mathrm{CAT(0)}$ cube complex and let $F\subseteq X$ be a
flat. Let $h$ be a hyperplane in $X$ intersecting $F$, and let $h^{+}$ and $h^{-}$ be the halfspaces
of $h$. Then either $F\subseteq h$, or $h\cap F$ is a
codimension-$\mathrm{1}$ flat in $F$. In the latter case, both $h^{+}\cap
F$ and $h^{-}\cap F$ are half-flats.
\end{lem}

\begin{proof} The carrier $N_h$ of $h$, which is its neighbourhood, has form
$N_h=h\times [0,1]$. Thus if $F\nsubseteq h$, then $h\cap F$ is a codimension-1
submanifold of $F$. Moreover, the intersections $h\cap F$,
$h^{+}\cap F$, and $h^{-}\cap F$ are convex, thus the lemma
follows.
\end{proof}

\begin{lem}
\label{hyperplane intersect ray} Let $h$ be a hyperplane in a
$\mathrm{CAT(0)}$ cube complex $X$. Suppose that $l$
is a geodesic ray in $X$ starting in $h$. If $l\nsubseteq h$, then
there exists another hyperplane~$h'$ in~$X$
intersecting $l$ and disjoint from $h$.
\end{lem}

\begin{proof}
Let $N_h$ be the carrier of $h$. Let $B$ be the first cube
outside $N_h$ whose interior is intersected by $l$. We claim that there is a hyperplane $h'$ intersecting $B$ and disjoint from
$h$. Indeed, pick a vertex $v\in N_h\cap B$ and let $e$ be an edge of $B$
containing $v$. If the hyperplane dual to $e$ intersects $h$,
then $e\subset N_h$. If this holds for any $e$, then $B\subset
N_h$ by the convexity of $N_h$, which yields a contradiction.
This justifies the claim.

By the claim, there a hyperplane $h'$ intersecting $B$ and disjoint from
$h$. It remains to prove that $l$ intersects $h'$. Otherwise, since $l$ intersects the interior of the carrier~$N_{h'}$, we have that $l$ is contained in $N_{h'}$.
Since $l$ starts at $h$, we have that $h$ intersects~$N_{h'}$
and hence it also intersects $h'$, which is a contradiction.
\end{proof}

We will also use a consequence of a result of Haglund
\cite[Theorem 2.28]{haglund2008finite}.

\begin{thm}
\label{thm:Haglund} Let $X$ be a hyperbolic $\mathrm{CAT(0)}$
cube complex. Then any quasi-isometrically embedded subspace of
$X$ is at finite Hausdorff distance from its combinatorial
convex hull.
\end{thm}

In the following theorem we generalise our results from
\cite[Section 3]{hagen-przytycki}. Here $d_{\mathrm{Haus}}$
denotes the Hausdorff distance.

\begin{thm}
\label{flat} Let $X$ be a $\mathrm{CAT(0)}$ cube complex of
asymptotic rank $n$ and let $F\subseteq~X$ be an $n$-flat. Let
$Y$ be the combinatorial convex hull of $F$. Then
$d_{\mathrm{Haus}}(F,Y)<~\infty$.
\end{thm}

\begin{proof}
If $F$ is contained in the carrier $N_h=h\times [0,1]$ of a
hyperplane $h$, then we can replace~$X$ by $h$ and $F$ by its
projection to $h$. The combinatorial convex hull $Y$ of $F$
equals $Y'\times [0,1], Y'\times \{0\}$, or $Y'\times \{1\}$,
where $Y'$ is the combinatorial convex hull of the projection
of $F$ to $h$. Henceforth we can and will assume that $F$ is
not contained in the carrier of any hyperplane.

Let $\h$ be the collection of hyperplanes intersecting $F$. We
define a \textit{pencil of hyperplanes} to be an infinite
collection of mutually disjoint hyperplanes
$\{h_{i}\}_{i=-\infty}^{\infty}$ such that for each $i$,
$\{h_{j}\}_{j=-\infty}^{i-1}$ and $\{h_{j}\}_{j=i+1}^{\infty}$
are in different halfspaces of $h_i$. It follows from
Lemma~\ref{hyperplane intersect flat} that every pencil of
hyperplanes in $\h$ intersects $F$ in a collection of parallel
family of codimension-1 flats. A collection of pencils of
hyperplanes in~$\h$ is \textit{independent} if their
corresponding normal vectors are linearly independent in
$F=\R^n$.

Let $\{P_{i}\}_{i=1}^{m}$ be a maximal collection of pairwise
independent pencils in $\h$. We claim that $m=n$ and that
$\{P_{i}\}$ is independent. Suppose first $m>n$. Note that if
two pencils $P,P'\subseteq\h$ are independent, then every
hyperplane in $P$ intersects every hyperplane in $P'$. This
gives rise to a quasi-isometric embedding of $\R^{m}$ into~$X$,
contradicting the bound on the asymptotic rank of $X$. If $m<n$
or if $m=n$ but $\{P_{i}\}$ is dependent, then there is a
geodesic line $l$ in $F$ parallel to $h\cap F$ for all
hyperplanes~$h$ in all $P_i$. Using Lemma~\ref{hyperplane
intersect ray}, we can then produce a new pencil $P$ formed of
some hyperplanes intersecting~$l$. Since $P$ is independent
from each $P_i$, this contradicts the maximality of~$m$. This
justifies the claim that $m=n$ and $\{P_{i}\}$ is independent.

For $1\le i\le n$, choose $h_{i}\in P_{i}$ and let
$F_{i}=h_{i}\cap F$. We will prove that for any hyperplane
$h\in \h$, there exists $F_{i}$ such that $h\cap F$ is parallel
(possibly equal) to~$F_{i}$. Otherwise, choose a geodesic line
$l$ in $F$ transverse to $h\cap F$. By Lemma~\ref{hyperplane
intersect ray}, $h$ is contained in a pencil $P_h$ of
hyperplanes intersecting $l$. Note that $P_h$ is independent
from each $P_i$, contradicting the maximality of~$m$.

Let $\h_{i}\subseteq \h$ be the collection of hyperplanes whose
intersection with $F$ is parallel to~$F_i$. The above
discussion implies $\h=\bigsqcup_{i=1}^{n}\h_{i}$. Moreover,
for $i\neq j$, every hyperplane in $\h_i$ intersects every
hyperplane in $\h_j$. Let $Y$ be the combinatorial convex hull
of $F$. Since we assumed that $F$ is not contained in the
carrier of any hyperplane, the hyperplanes in $Y$ are exactly
the intersections with $Y$ of the hyperplanes in $\h$. Two
hyperplanes of $Y$ intersect if and only if the corresponding
hyperplanes in $\h$ intersect. Hence by Lemma~\ref{product
decomposition cube}, we have a product decomposition
$Y=Y_1\times \cdots \times Y_n$.

Let $\pi_{i}:Y\to Y_i$ be the coordinate projections. Let
$l_{i}=\bigcap_{j\neq i}F_j$, which is a geodesic line in $F$.
Note that for $j\neq i$ we have $l_{i}\subseteq F_{j}\subseteq
h_{j}$ and hence the projection $\pi_{j}(l_{i})$ is a single
point. Thus the restriction of $\pi_{i}$ to $l_{i}$ is an
isometric embedding. It follows that $F=\pi_1(l_1)\times\cdots
\times \pi_1(l_n)$. Moreover, since $\pi_i(l_i)=\pi_i(F)$, each
$Y_i$ is the combinatorial convex hull of $\pi_i(l_i)$, since
otherwise we could pass to a smaller convex subcomplex
containing $F$.

Since each of $Y_i$ contains a line and their product has
asymptotic rank $\le n$, by Theorem~\ref{thm:asrk}(1) each
$Y_i$ has asymptotic rank $1$. By Theorem~\ref{thm:asrk}(2)
each $Y_{i}$ is hyperbolic. Thus by Theorem~\ref{thm:Haglund},
we have $d_{\mathrm{Haus}}(\pi_i(l_i), Y_i)<\infty$, and
consequently $d_{\mathrm{Haus}}(F,Y)<\infty$.
\end{proof}

While we will not need it in the remaining part of the paper,
from the proof above we can deduce the following interesting result which
concerns flats that are not necessarily of top rank.

\begin{cor}
\label{arbitrary flat} Let $X$ be a $\mathrm{CAT(0)}$ cube
complex and let $F\subseteq X$ be a flat. Let $Y\subseteq X$ be
the combinatorial convex hull of $F$. Then $Y$ has a natural
decomposition $Y=Y_{1}\times\cdots\times~Y_n\times~K$ such
that:
\begin{enumerate}
\item $n\ge \dim(F)$ and $K$ is a cube.
\item each $Y_{i}$ contains an isometrically embedded copy
    of $\R$ that is the projection of a geodesic line in
    $F$.
\item no $Y_i$ contains a facing triple of hyperplanes,
    that is, a collection of three disjoint hyperplanes
    such that none of them separates the other two.
\end{enumerate}
\end{cor}

Roughly speaking, (3) means that $Y_{i}$ do not
\textquotedblleft branch\textquotedblright.

\subsection{Product of hyperbolic groups}\label{sec:product}
\begin{defn}
Let $X$ be a CAT(0) cube complex. A group $H\le Aut(X)$ is
\textit{convex cocompact} if there is a convex subcomplex
$Y\subseteq X$ that is \emph{$H$-cocompact}, meaning that $H$
preserves $Y$ and acts on it cocompactly.
\end{defn}

\begin{lem}
\label{minimal invariant} Let $X$ be a $\mathrm{CAT(0)}$ cube
complex and let $H\le Aut(X)$ be convex cocompact. Then there
exists a minimal $H$-invariant convex subcomplex. Moreover, any
minimal $H$-invariant convex subcomplex is $H$-cocompact and
any two minimal $H$-invariant convex subcomplexes are parallel.
\end{lem}

\begin{proof}
Let $Y\subseteq X$ be an $H$-cocompact convex subcomplex. Let
$\p$ be the poset of $H$-invariant convex subcomplexes in $Y$.
For the first assertion, by the Kuratowski--Zorn Lemma, it
suffices to show that every descending chain of elements
$\{Y_\lambda\}_\lambda\subseteq\p$ has a lower bound, or
equivalently that their intersection is nonempty. Let
$K\subseteq Y$ be compact such that $HK=Y$. Then each $K\cap
Y_\lambda$ is nonempty, and by compactness of $K$ so is their
intersection.

For the second and third assertion, let
$Y_{\mathrm{min}}\subseteq Y$ be a minimal element of $\p$ and
let $Y'$ be any other minimal $H$-invariant convex subcomplex.
Let $(V,V')=\G(Y_{\mathrm{min}},Y')$. Then both $V$ and $V'$
are $H$-invariant. By Lemma~\ref{gate}(1) both $V$ and $V'$ are
convex subcomplexes, hence from minimality of
$Y_{\mathrm{min}}$ and $Y'$ we have $V=Y_{\mathrm{min}}$ and
$V'=Y'$. Moreover, by Lemma~\ref{gate}(2) we have that $Y'$ is
$H$-equivariantly isometric to $Y_{\mathrm{min}}$ and thus it
is $H$-cocompact.
\end{proof}

\begin{thm}
\label{product of hyperbolic}
Let $X$ be a locally finite $\mathrm{CAT(0)}$ cube complex of asymptotic rank~$n$. Let $H\le Aut(X)$ be a subgroup
satisfying
\begin{enumerate}
\item $H=H_1\times\cdots\times H_n$, where each $H_i$ is an infinite hyperbolic
group, and
\item for some (hence any) point $x\in X$ the orbit map $h\to h\cdot x$ from $H$ to $X$ is a quasi-isometric
embedding.
\end{enumerate}
Then $H$ is convex cocompact. More precisely, if among $H_i$
exactly $\{H_{i}\}_{i=1}^m$ are not virtually $\Z$, then there
is a convex subcomplex $Y\subseteq X$ with a cubical product
decomposition $Y=Y_0\times \prod_{i=1}^{m}Y_{i}$ such that
\begin{enumerate}[(i)]
\item $Y$ is $H$-cocompact, and the action
    $H\curvearrowright Y$ respects the product
    decomposition, and
\item the induced action of $\prod_{i=m+1}^{n}H_{i}$ on $Y_0$ is proper
and cocompact, in particular $Y_0$ is quasi-isometric to
$\R^{n-m}$, and
\item for any pair $i\neq j$ with $1\le j\le m$ and $1\le
    i\le n$, the induced action $H_{i}\curvearrowright
    Y_{j}$ is \emph{almost trivial}, i.e.\ by isometries at
    uniformly bounded distance from the identity.
\end{enumerate}
\end{thm}

In the proof we need the notion of coarse intersection. Let $X$
be a metric space and let $N_{R}(Y)$ be the $R$-neighbourhood
of a subspace $Y\subseteq X$. A subspace $V\subseteq X$ is the
\textit{coarse intersection} of $Y_1$ and $Y_2$ if~$V$ is at
finite Hausdorff distance from $N_R(Y_1)\cap N_{R}(Y_2)$ for
all sufficiently large $R$. For example, in Lemma~\ref{gate},
in view of its part (3), the gates $V_1,V_2$ are the coarse
intersections of $Y_{1}$ and $Y_2$. However, in general the
coarse intersection of two subsets might not exist.

\begin{lem}[{\cite[Lemma
2.2]{mosher2004quasi}}]
\label{lem:coarse}
Let $X$ be a finitely generated group with word metric. Then the coarse intersection of a
pair of subgroups is well-defined and represented
by their intersection.
\end{lem}

See
\cite[Chapter~2]{mosher2004quasi} for more discussion on coarse
intersection.

\begin{proof}[Proof of Theorem~\ref{product of hyperbolic}]
We first prove that $H$ is convex cocompact, which we do by the
induction on~$m$. Consider first the case $m=0$. By
\cite{Hagsemisimple}, $H$ acts on $X$ be semi-simple
isometries. By the Flat Torus Theorem \cite[Chapter~II.7]{MR1744486}, $H$ acts cocompactly on an $n$-flat
$F\subseteq
X$. By Theorem~\ref{flat}, the combinatorial convex hull $Y$ of
$F$ is at finite Hausdorff distance from $F$. Since $X$ is
locally finite, $Y$ is $H$-cocompact, as desired.

Suppose now that $m\geq 1$. Let $H'=\prod_{i\neq m}H_i$. We
first prove that the group $H'$ is convex cocompact. Choose a
subgroup $Z\le H_m$ isomorphic to $\Z$ and choose $h\in H_m$
such that the coarse intersection of $hZ$ and $Z$ is bounded.
Let $G=H'\times Z\subset H$. By induction assumption, there
exists a $G$-cocompact convex subcomplex $U\subset X$. Let
$V\subset U$ be the gate with respect to $h\cdot U$. Note that
both $U$ and $h\cdot U$ are $H'$-invariant, so $V$ is
$H'$-invariant. By Lemma~\ref{gate}(3), $V$ is the coarse
intersection of $U$ and $h\cdot U$. Hence, by Lemma~\ref{lem:coarse} applied to $G$ and $hGh^{-1}$, the action $H'\acts V$ is cocompact.

By Lemma~\ref{minimal invariant}, there exists a minimal
$H'$-cocompact convex subcomplex, for which we keep the
notation $V$. Then for any $h\in H_m$, the translate $h\cdot V$
is minimal $H'$-invariant, hence parallel to $V$ by
Lemma~\ref{minimal invariant}. Let $P_V=V\times V^{\perp}$ be
the combinatorial parallel set of $V$ (see Lemma \ref{parallel
set}). We have that $P_V$ is $H$-invariant. Moreover, since $V$
is $H'$-invariant, there are induced actions $H\acts V^{\perp}$
and $H_m\acts V^{\perp}$.

Choose a point $v\in V$. Let $\psi:H_m\to V^{\perp}$ be the composition
of the orbit map $h\to h\cdot v$ with the coordinate
projection. We claim that $\psi$ is a quasi-isometric
embedding. This follows from assumption~(2) and the estimates
below, where $\sim$ means equality up to a uniform
multiplicative and additive constant. Namely, for any
$h_1,h_2\in H_m$ we have:
\begin{equation*}
d_{H_m}(h_1,h_2)\sim d_H(h_1 H',h_2H')\sim d_X(h_1\cdot V,h_2 \cdot V)=d_{V^{\perp}}(\psi(h_1),\psi(h_2))
\end{equation*}

By Theorem~\ref{thm:asrk}, since $V$ contains an isometrically
embedded copy of $\R^{n-1}$, the asymptotic rank of $V^{\perp}$
is $\le 1$, and hence $V^{\perp}$ is hyperbolic. Let $V_m\subseteq
V^{\perp}$ be the combinatorial convex hull of $\psi(H_m)$.
Then $d_{\mathrm{Haus}}(V_m,\psi(H_m))<\infty$ by
Theorem~\ref{thm:Haglund}. Moreover, $V_m$ is $H$-invariant
under the action $H\acts V^{\perp}$ since $\psi(H_m)$ is
invariant under $H$. Thus $H$ acts cocompactly on the convex
subcomplex $V\times V_m\subseteq P_V$. Notice that since
$H'\curvearrowright \psi(H_m)$ is trivial, the action
$H'\curvearrowright V_m$ is almost trivial.

By now we already know that $H$ is convex cocompact. As for
properties (i)---(iii), if $m=1$, then it suffices to take $Y_0=V$ and $Y_1=V_1$.
If $m\geq 2$, to obtain
the required decomposition, we consider $X'=V\times
V_m,\ H''=\prod_{i\neq (m-1)}H_i$ and we repeat the previous
argument. This gives rise to an $H$-cocompact convex subcomplex
$V'\times V_{m-1}\subseteq V\times V_m$, where $V'$ is a
minimal $H''$-cocompact convex subcomplex. Since $V_m$ is
contained in some $R$-neighbourhood of a $V'$, the intersection
$V_{m-1}\cap V_m$ is compact. Moreover, $V'$ and $V_{m-1}$
admit cubical product decompositions $V'=(V'\cap
V)\times(V'\cap V_m)$ and $V_{m-1}=(V_{m-1}\cap
V)\times(V_{m-1}\cap V_m)$, thus $V'\times V_{m-1}=(V'\cap
V)\times(V'\cap V_m)\times(V_{m-1}\cap V)\times(V_{m-1}\cap
V_m).$ The $H$-action respects the above decomposition.
Moreover, the induced action $H'\curvearrowright (V'\cap V_m)$
is almost trivial and the induced action $H''\curvearrowright
(V_{m-1}\cap V)$ is almost trivial.
If $m=2$, then we take $Y_1=V_1\cap V,\ Y_2=V'\cap V_2,$ and
$Y_0=(V\cap V')\cup (V_1\cap V_2)$. If $m\geq 3$, then we let $X''=V'\times
V_{m-1}, H'''=\prod_{i\neq (m-2)}H_i$ and we repeat the
previous process to obtain further product decomposition. We
run this process $m$ times, obtaining the required decomposition as the result of the last step. In each step, we possibly
get nontrivial compact factors similar to $V_{m-1}\cap V_m$. We
absorb all these compact factors into the factor $Y_0$ (we can also discard them).
\end{proof}

\section{Artin groups}\label{sec:artin groups}
\subsection{Background on Artin groups}
Let $A$ be an Artin group with defining graph~$\Gamma$, and
generators $S$. Let $W$ be the Coxeter group defined by
$\Gamma$. For any $T\subseteq S$ let $W_T$ (respectively $A_T$)
be the \emph{special subgroup} of~$W$ (respectively $A$)
generated by $T$. The special subgroup $W_T$ is naturally
isomorphic to the Coxeter group defined by the subgraph
$\Gamma_T$ induced on $T$ \cite{Bourbaki1968}. Similarly, by
\cite{Van1983homotopy} the special subgroup $A_T$ of $A$ is
naturally isomorphic to the Artin group defined by $\Gamma_T$.

\begin{lem}[{\cite[Theorem 1.1]{CharneyParis2014}}]
\label{lem:special convex}
Special subgroups of Artin groups are convex with respect to
the word metric defined by standard generators.
\end{lem}

A subset $T\subseteq S$ is \emph{spherical} if the special
subgroup $W_T$ is finite. The \emph{dimension} of the Artin
group $A$ is the maximal cardinality of a spherical subset of
$S$.

The following is a consequence of \cite{charney1995k} and
\cite[Corollary 1.4.2]{charney138finite}.
\begin{thm}
\label{K(pi,1)} Let $A$ be an Artin group of dimension $n$.
Suppose that
\begin{enumerate}[(A)]
\item $n\leq 2$, or
\item every clique $T$ in $\Gamma$ is spherical.
\end{enumerate}
Then there is a finite $n$-dimensional cell complex that is a
$K(A,1)$.
\end{thm}

\subsection{Two-generator Artin groups}

We start with the description of most two-generator Artin
groups as virtually $F_k\times \Z$, where $F_k$ is the free
group with $k$ generators.

\begin{lem}\label{lem:free times z}
Let $A$ be an Artin group with defining graph $\Ga$ a single
edge labelled by $n>2$. Then
\begin{enumerate}
\item $A$ has a finite index subgroup of form $F_k\times
    \Z$ with $k\geq 2$, and
\item no power of one of the two standard generators lies in the $\Z$ factor.
\end{enumerate}
\end{lem}
\begin{proof}
By \cite{BradyMccammond2000} (or by our proof of
Theorem~\ref{thm:constr}) $A$ acts freely and cocompactly on a
product of a tree and a line, where a central element acts as a
translation in the line factor. By \cite[Theorem~II.6.12]{MR1744486} $A$ virtually decomposes as $A'\times \Z$. The
induced action of $A'$ on the tree factor has finite vertex
stabilisers so by Bass-Serre theory $A'$ is a graph of finite
groups, in particular $A'$ is virtually free, justifying~(1).
Part~(2) follows from the fact that standard generators act
hyperbolically on the tree factor.
\end{proof}

Throughout this section by $\bar x$ we denote the inverse of
$x$. By $x^z$ we denote the conjugate $\bar z x z$.

Let $A_n = \angled{a,b\mid \underbrace{aba\dots}_{n} =
\underbrace{bab\dots}_n}$. Denote $\underbrace{aba\dots}_{n} =
\underbrace{bab\dots}_n$ by $\Delta$. Let $A'_n$ be the kernel
of the homomorphism sending each generator to the generator of
$\Z/2$ i.e.\ the subgroup consisting of all words of even
length. The group $A'_n$ is generated by the elements: $r = ab,
s = a\bar b, t = \bar ab,$ since any word of even length can be
written as a product of these elements and their inverses.
If~$\phi$ is a word in an alphabet $\Lambda$, and
$x\in\Lambda$, then we denote by $\mathrm{Exp}_x(\phi)$ the sum
of all the exponents at $x$ in $\phi$.

By direct computation we immediately establish the following:
\begin{lem}
If $n$ is odd, then the conjugation by $\Delta$ is an order two
automorphism sending $s\mapsto \bar s, t\mapsto \bar t,
r\mapsto q$, where $q = ba = \bar s r \bar t$. In particular,
$\Delta^2$ is a central element.

If $n$ is even, then $\Delta$ is a central element.
\end{lem}
Let $z$ be the element $\Delta^2$ for $n$ odd and the element $\Delta$ for $n$ even.

\begin{lem}\label{lem:b^n}
If $n$ is odd, then we have
\[
b^n = \phi(s,t,r)\Delta,
\]
where $\mathrm{Exp}_r(\phi) = 0$. \end{lem}
\begin{proof}
Consider the following word $\phi$ expressed as a product of
terms indexed by decreasing $i$:
\[
\phi(s,t,r) = \bar s \prod_{i = \frac{n-3}{2}}^{0}  \bar t^{r^{i}}
\]
Since $r^i$ appear in the expression defining $\phi$ only as
elements that we conjugate by, we have $\mathrm{Exp}_r(\phi) =
0$.

To verify that $b^n=\phi\Delta$, note that
\[
\phi=\bar s \prod_{i = \frac{n-3}{2}}^{0} \bar r^{i}\bar tr^{i} =
\bar s (\bar r^{\frac{n-3}{2}}\bar t r^{\frac{n-3}{2}}) (\bar r
^{\frac{n-3}{2}-1}\bar t r^{\frac{n-3}{2}-1})\dots (\bar r\bar
t r )\bar t=\bar s \bar r^{\frac{n-1}{2}}(r\bar t )^{\frac{n-1}{2}}.
\]
Since $\bar s\bar
r^{\frac{n-1}{2}} = b\bar a (\bar b\bar a)^{\frac{n-1}{2}} =b
\bar \Delta$ and $r\bar t \Delta =\Delta qt= \Delta b^2$, we have
\[
\phi(s,t,r)\Delta = \bar s \bar r^{\frac{n-1}{2}}\Delta b^{n-1} = b\bar \Delta \Delta b^{n-1} = b^n.
\]
\end{proof}

\begin{cor}
\label{cor:formula_for_b}
 If $n$ is odd, we have
\[
b^{2n}\bar{z} \in [A'_n,A'_n].\]
\end{cor}
\begin{proof}
We have
\[
b^{2n} = \phi(s,t,r)\Delta\phi(s,t,r)\Delta = \phi(s,t,r)\phi(\bar s, \bar t, q)z.
\]
Denote the word $\phi(s,t,r)\phi(\bar s, \bar t, q)$ by
$\psi(s,t,r,q)$. By Lemma~\ref{lem:b^n}, we have
$\mathrm{Exp}_r(\psi) =\mathrm{Exp}_q(\psi) = 0$. We also have
$\mathrm{Exp}_s(\psi) = \mathrm{Exp}_t(\psi) = 0$ since the
total exponents of $s$ and $t$ in $\phi(s,t,r)$ are equal to
the total exponents of $\bar s$ and $\bar t$ in $\phi(\bar s ,
\bar t, q)$, respectively. Thus $\psi \in [A_n',A_n']$.
\end{proof}
Corollary~\ref{cor:formula_for_b} does not hold for $n$ even, since in that case $\Delta$ is a central element.

\subsection{Surface lemma}

The following lemma will allow us to utilise the preceding result
when discussing finite index subgroups of $A_n$.

\begin{lem}\label{lem:surface lemma}
Let $G$ be a finitely generated group and let $z\in G$ be
central. Let~$H$ be a finite index normal subgroup of $G$, and
let $h\in H\cap z[G,G]$. Then for any homomorphism $\rho:H\to
\Z$ such that $\rho(\angled{z}\cap H) \neq \{0\}$, there exist
a positive integer~$m$ and $g\in G$ with $\rho((h^m)^g)\neq 0$.
\end{lem}

\begin{proof}
Let $X$ be a presentation complex for $G$. Let $S$ be an
oriented surface with connected $\partial S$ and basepoint
$s\in\partial S$, mapping to $X$, such that on the level of
fundamental groups $\partial S\mapsto h\bar z$. Let $\widehat
X$ be the finite cover of $X$ corresponding to $H$ and let
$\widehat S$ be a finite cover of $S$ such that $\widehat S \to
S \to X$ lifts to $\widehat S\to \widehat X$. Choose a system~$\Sigma$ of nonintersecting arcs that join the basepoint of
$\widehat S$ to the other preimages of $s$, one for each of the
boundary components of $\widehat S$. Consider the surface $S'$
obtained from~$\widehat S$ by cutting along the arcs
of~$\Sigma$, and the mapping $S'\to \widehat X$ that factors
through~$\widehat S$. Then, as the boundary of a surface,
$\partial S'$ is mapped to an element $f\in H=\pi_1(\widehat
X)$ contained in $[H,H]$. The arcs of $\Sigma$ map to paths in
$\widehat X$ that project to closed paths in $X$ corresponding
to some $g_i\in G$. Thus we have $f=\prod_{i=1}^{q}
(h^{m_i})^{g_i}\bar z^{M}$, where $m_i\geq 1$ with $M=\sum
m_i$.

Since $H$ is normal, each $(h^{m_i})^{g_i}$ lies in $H$. We
have $\rho(\prod_{i=1}^{q}(h^{m_i})^{g_i}) = \rho(z^{M})\neq
0$. That means that there is at least one element
$(h^{m_i})^{g_i}$ such that $\rho((h^{m_i})^{g_i}) \neq 0$.
\end{proof}

\begin{cor}\label{cor:odd n}
Let $n$ be odd and let $H$ be a finite index normal subgroup of
$A_n'$. Then for any homomorphism $\rho:H\to \Z$ such that
$\rho(\angled{z}\cap H)\neq \{0\}$, there exist a positive
integer $m$ and $g\in A'_n$ such that $b^m\in H$ and
$\rho((b^m)^{g})\neq 0$.
\end{cor}

\begin{proof}
Let $k$ be large enough so that $b^{2nk}\in H$. By
Corollary~\ref{cor:formula_for_b}, we can apply
Lemma~\ref{lem:surface lemma} with $G=A_n',h=b^{2nk},$ and
$z^k$ in the role of $z$.
\end{proof}

\begin{cor}\label{cor:even n}
Let $n$ be even and let $H$ be a finite index normal subgroup
of $A_n$. Then for any homomorphism $\rho:H\to \Z$ such that
$\rho(\angled{z}\cap H)\neq \{0\}$, there exist a positive
integer $m$ and $g\in A_n$ such that at least one of
$(a^m)^{g}$ and $(b^m)^{g}$ lies in $H$ and is not mapped to
$0$ under $\rho$.
\end{cor}

\begin{proof}
Let $k = \frac{n}2k'$ be a nonzero integer such that
$a^k,b^k\in H$. Since $z^{k'} = (ab)^k$, we have
\[a^kb^k\in z^{k'}[A_n,A_n].\]
By Lemma~\ref{lem:surface lemma}, we have $m>0$ and $g\in A_n$
such that $\rho\big((a^kb^k)^m)^{g}\big)\neq 0$. Let
$f=(a^{k})^g$ and $h=(b^{k})^g$. We have $(fh)^{m} \in
f^{m}h^{m}[H,H]$. Thus $\rho(f^{m}h^{m}) \neq 0$ and so at
least one of $f^{m} = (a^{km})^{g}$ and $h^{m} = (b^{km})^{g}$
is not mapped to $0$ under $\rho$.
\end{proof}

\section{The main theorem}\label{sec:main}

In this section we prove Theorem~\ref{thm:main}. The
implication (i)$\Rightarrow$(ii) is obvious.

\subsection{Implication (iii)$\Rightarrow$(i)}

\begin{thm}
\label{thm:constr}
Let $A$ be an Artin group with each connected component of the
defining graph:
\begin{itemize}
\item a vertex, or an edge, or else \item all interior
    edges labeled by 2 and all leaves labelled by even
    numbers.
\end{itemize}
    Then $A$ is the fundamental group of a nonpositively curved cube complex.
\end{thm}

\begin{proof}
We assume without loss of generality that $\Ga$ is connected,
since if $\Gamma$ has more connected components, then $A$ is
the fundamental group of the wedge of the complexes obtained
for its connected components.

If $\Gamma$ is a single vertex, then $A$ is the fundamental
group of a circle.

If $\Ga$ is a single edge labelled by an odd $n$, then let $K_n$ be the cube complex described in the figure below.
\begin{center}
\includegraphics[scale=0.5]{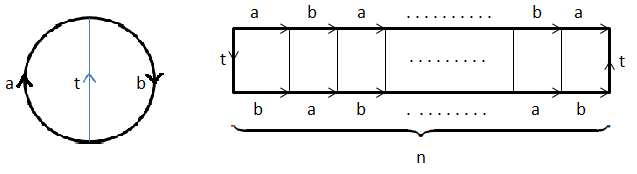}
\end{center}
On the left side we see the 1-skeleton of $K_n$ consisting of
three edges labelled by $a,b,t$, and the right side indicates
how to attach the unique $2$-cell (subdivided into $n$ squares)
along its boundary path $\underbrace{ab\dots a}_{n}\bar t
\underbrace{\bar b \bar a\dots \bar b}_{n}\bar t$. It is easy
to check that the link of each of the two vertices in $K_n$ is
isomorphic to the spherical join of two points with $n$ points,
hence $K_n$ is nonpositively curved. By collapsing the $t$-edge
we obtain the presentation complex for the standard
presentation of $A$, so $\pi_{1}(K_n)=A$. We learned this
construction from Daniel Wise.

If $\Ga$ is a single edge labelled by an even $n$, let $x=ab$.
The group $A$ is then presented as $\angled{a,x\mid
ax^{n/2}=x^{n/2}a}$. Let $K_{n,a}$ be the cube complex
described in the figure below.
\begin{center}
\includegraphics[scale=0.5]{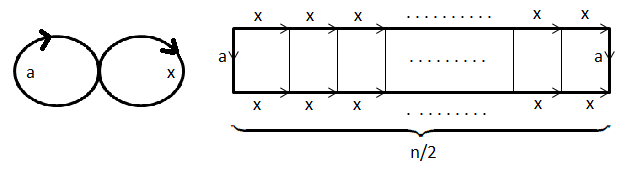}
\end{center}
One can check that the link of the unique vertex in $K_{n,a}$
is isomorphic to the spherical join of two points with $n$
points, hence $K_{n,a}$ is nonpositively curved. It is clear
that $\pi_1(K_{n,a}) = A$.

Similarly if we let $y=ba$, then $A$ can be presented as
$\angled{b,y\mid by^{n/2}=y^{n/2}b}$. We define $K_{n,b}$ in a
similar way. Note that the $a$-circle in $K_{n,a}$ is a locally
convex subcomplex, so is the $b$-circle in $K_{n,b}$.

If $\Ga$ contains more than one edge, then let
$\Ga'\subseteq\Ga$ be the nonempty subgraph induced on all the
vertices that have at least two neighbours. Thus the edges of
$\Ga'$ are precisely the interior edges and by the hypothesis
they are labelled by $2$. Hence $A_{\Ga'}$ is a right-angled
Artin group. The \emph{Salvetti complex} $S(\Ga')$ is the
nonpositively curved cube complex obtained from the
presentation complex of $A_{\Ga'}$ by adding the missing cubes
of higher dimension (see \cite{charney2007introduction}). Let
$\{(s_i,t_i)\}_{i=1}^{k}$ be the collection of leaves of $\Ga$
with $s_{i}\in\Ga'$. Let $n_i$ be the label of the edge
$(s_i,t_i)$, which is even. Let $K$ be the amalgamation of
$\{K_{n_i,s_i}\}_{i=1}^{k}$ and $S(\Ga')$ along the
$s_i$-circles. Then $\pi_{1}(K)=A$ and it follows from
\cite[Proposition II.11.6]{MR1744486} that $K$ is nonpositively
curved.
\end{proof}

\subsection{Implication (ii)$\Rightarrow$(iii)}

\begin{thm}
\label{thm:2dim}
Let $A$ be a $2$-dimensional Artin group. If $A$ is virtually cocompactly cubulated,
then each connected component of the defining graph of $A$
    is either
\begin{itemize}
\item a vertex, or an edge, or else \item
all its interior
    edges are labeled by $2$ and all its leaves are
    labelled by even numbers.
\end{itemize}
\end{thm}

\begin{proof}
Suppose that there exists a finite index subgroup $\hat A\le A$
that acts properly and cocompactly by combinatorial
automorphisms on a CAT(0) cube complex $X$. Without loss of
generality, we assume that $\hat A$ is normal in $A$. It
suffices to prove:
\begin{enumerate}
\item no edge of $\Ga$ has an odd label, unless it is an
    entire connected component, and
\item no interior edge of $\Ga$ has an even label $\geq 4$.
\end{enumerate}

Let us first prove (1). Suppose to the contrary that $\Gamma$
has an edge $(a,b)$ with odd label and another edge $(b,c)$.
Let $A_{ab}$ be the special subgroup generated by $a$ and~$b$,
and let $A'_{ab}$ be its index-two subgroup from the previous
section. Let $\hat A_{ab}=F_{k}\times \Z$ be a finite index
subgroup of $A'_{ab}\cap \hat A$ guaranteed by
Lemma~\ref{lem:free times z}(1). We can also assume that $\hat
A_{ab}$ is normal in $A'_{ab}$. Similarly, let $A_{bc}$ be the
special
subgroup generated by~$b$ and~$c$, and let $\hat A_{bc} = F_l\times
\Z$ be a finite index subgroup of $A_{bc}\cap \hat A$. Note
that the edge $(b,c)$ might be labelled by $2$ and then $l=1$.

Since $\hat A$ is a CAT(0) group, we can speak of its
asymptotic rank. By Theorem~\ref{K(pi,1)}(A), there exists a
finite $2$-dimensional cell complex that is a $K(A,1)$. Thus by
Lemma~\ref{rank}, the asymptotic rank of $\hat{A}$ is $\leq 2$
and so is the asymptotic rank of~$X$. The subgroup $A_{ab}$ is
convex with respect to the standard generators of~$A$ by
Lemma~\ref{lem:special convex} and so $\hat A_{ab}$ is
quasi-isometrically embedded in $\hat A$. We can thus apply
Theorem~\ref{product of hyperbolic} to find a convex subcomplex
$Y_{ab}$ that is $\hat A_{ab}$-cocompact. Moreover, there is a
cubical product decomposition $Y_{ab}=V_{ab}\times H_{ab}$ such
that the action of~$\hat A_{ab}$ respects this decomposition,
the vertical factor $V_{ab}$ is quasi-isometric to $\R$, and
the $\Z$ factor $Z$ of $\hat A_{ab}$ acts almost trivially on
$H_{ab}$.

Consider $\mathrm{Min}(Z) =\R\times V_0\subseteq V_{ab}$ for the induced action of $Z$, where $\R$ is an axis of~$Z$.
Since $Z$ is contained in the centre of $\hat A_{ab}$, we have an induced
action of $\hat A_{ab}$ on $\R\times V_0$ respecting this decomposition.
The factor $V_0$ is bounded, so $V_0$ contains a fixed-point of
the action of~$\hat A_{ab}$. Thus $\R\times V_0$ contains an $\hat A_{ab}$--invariant line~$l$.
Let $\rho:\hat A_{ab}\to \mathrm{Isom}(l)$ be the induced map. Note that $\rho(\hat A_{ab})$ does not flip the ends of~$l$.
Moreover, since $V_{ab}$ is a cube complex, the translation lengths on $l$ are discrete.
This gives rise to a homomorphism $\rho:\hat A_{ab}\to \Z$ assigning to each
element of $\hat A_{ab}$ its translation length on $l$. Note that $\rho(Z)\neq 0$. By Corollary~\ref{cor:odd n}
applied to $H=\hat A_{ab}$, there exists a nonzero integer $m$
and $g\in A'_{ab}$ such that $\rho((b^m)^g)\neq 0$.

By normality of $\hat A$, we have $(\hat A_{bc})^g\le \hat A$.
Let $Y_{bc}$ be a convex $(\hat A_{bc})^{g}$-cocompact
subcomplex guaranteed again by Theorem~\ref{product of
hyperbolic}. By \cite{Van1983homotopy} we have $A_{ab}\cap
A_{bc}=A_b$, and hence the groups $\angled{b^{m}}^g$ and $\hat
A_{ab}\cap (\hat A_{bc})^{g}$ have a common finite index
subgroup~$B$. Let $Y\subset Y_{ab}$ be the gate with respect to
$Y_{bc}$. Then $Y$ is the coarse intersection of $Y_{ab}$ and
$Y_{bc}$ by Lemma~\ref{gate}(3). By Lemma~\ref{lem:coarse}, $Y$ is $B$-cocompact.

Since $Y$ is a convex subcomplex, it has a product structure
$Y=Y_V\times Y_H$ where $Y_V\subseteq V_{ab}$ and $Y_H\subseteq
H_{ab}$. We have $\rho(B)\neq 0$, so $Y_V$ is
unbounded. Since $Y$ is quasi-isometric to $\R$, the factor
$Y_H$ is bounded. Since $Z$ acts almost trivially on~$H_{ab}$,
any of its orbits in $Y_{ab}$ is at a finite Hausdorff distance
from $Y$. Hence $Z$ is commensurable with $B$. Thus there
exists an integer $j\neq 0$ such that $(b^g)^{j}\in Z$, and
hence $b^{j}\in Z$, contradicting Lemma~\ref{lem:free times
z}(2).

Let us now prove (2). Suppose that $\Gamma$ has edges
$(a,b),(b,c),$ and $(c',a)$ (here $c$ and $c'$ are possibly the
same), where $(a,b)$ has an even label $\geq 4$. Let $\hat
A_{ab}, \hat A_{bc}, \hat A_{c'a}$ be finite index subgroups of
$A_{ab}\cap \hat A,A_{bc}\cap \hat A,A_{c'a}\cap \hat A$,
respectively, that are isomorphic to a product of a free group
and $\Z$. Assume moreover that $\hat A_{ab}$ is normal in
$A_{ab}$. Let $Y_{ab}=V_{ab}\times H_{ab}$ be a convex $\hat A_{ab}$-cocompact
subcomplex, and let $\rho:\hat A_{ab}\to \Z$ be defined as before. By Corollary~\ref{cor:even n}, there exist a
nonzero integer $m$ and $g\in A_{ab}$ such that at least one of
$(a^{m})^g$ and $(b^{m})^g$ lies in $\hat A_{ab}$ and is not
mapped to~$0$ under~$\rho$. Without loss of generality we can
assume $\rho((b^{m})^g)\neq 0$. The rest of the argument is
identical as in the proof of~(1).
\end{proof}

\section{3-generator Artin groups}
\label{sec:three}

This section is devoted to the proof of Theorem~\ref{thm:3gen}.
Let $A$ be the three-generator Artin group with
$m_{ab}=3,m_{bc}=2$, and $m_{ac}=3,4,$ or $5$, and let $W$ be
the Coxeter group with the same defining graph. Consider a
longest word in $a,b,c$ which is a minimal length
representative of the element it represents in $W$. This word
represents also an element of $A$, which we call $\Delta$.

\begin{lem}
\label{lem:proj}
\begin{enumerate}[(i)]
\item The centre $Z$ of $A$ is generated by $\Delta^2$ for
    $m_{ac}=3$ and by $\Delta$ for $m_{ac}=4$ or $5$.
\item The intersections of $A_{ab}$ and $A_{bc}$ with $Z$
    are trivial.
\item In $A$ we have $A_{ab}\times Z\cap A_{bc}\times
    Z=A_b\times Z$.
\end{enumerate}
\end{lem}

\begin{proof}
Assertion (i) follows from \cite[Theorem 4.21]{Deligne}.

For (ii), let $\Delta_{ab}=aba$. By
\cite[Proposition~4.17]{Deligne}, each element of $A_{ab}$ is
represented by $\Delta_{ab}^{-k}\phi(a,b)$, where $\phi$ is a
positive word in $a,b$, and $k\geq 0$. If we had
$\phi(a,b)=\Delta_{ab}^k\Delta^l$ for some $l>0,k\geq 0$, then
by \cite[Theorem~4.14]{Deligne} this equality would also hold
in the Artin semigroup, contradicting the fact that $\Delta$ is
expressed as a positive word involving all $a,b,c$. The same
argument works for $A_{bc}$.

For (iii) we need to show $A_{ab}\times Z\cap A_{bc}\times
    Z\subseteq A_b\times Z$. Since $b$ and $c$ commute, it
    suffices to show that for each $m\neq 0$ we have $c^m\notin A_{ab}\times Z$. If $m_{ac}=3$,
    then this follows from a well known fact that $A/Z$ is the mapping class group
    of the four punctured disc, where $A_{ab}$ fixes a curve
    around the first three punctures and $c$ is a half-Dehn
    twist in a curve around the third and the fourth.

If $m_{ac}=4$ or $5$, assume for contradiction that $c^m=gz$,
for some $z\in Z$ and $g\in A_{ab}$. Thus $gc^m=g^2z=gzg=c^mg$.
Let $g=\Delta_{ab}^{-k}\phi(a,b)$, where $\phi$ is a positive
word in $a,b$, and $k\geq 0$ is even. Thus
$\phi(a,b)c^m\Delta^{k}_{ab}=\Delta^k_{ab}c^m\phi(a,b)$.

By \cite[Theorem~4.14]{Deligne} this equality also holds in the
Artin semigroup. The relation $acac=caca$ or $acaca=cacac$
involves on each side $2$ occurences of $c$ separated by an
occurence of $a$. The word $\phi(a,b)c^m\Delta^{k}_{ab}$ does
not contain such a subword, and this property is invariant
under the replacements $bc=cb$, $aba=bab$. Thus to pass from
$\phi(a,b)c^m\Delta^{k}_{ab}$ to $\Delta^k_{ab}c^m\phi(a,b)$
one can only use $bc=cb$, and $aba=bab$, which is the relation
defining~$A_{ab}$. Thus there is $l$ such that in $A_{ab}$ we
have $\phi(a,b)b^l=\Delta^k_{ab}$. Hence $g=b^{-l}$. Thus
$c^m=b^{-l}z$, contradicting assertion (ii).
\end{proof}

We also need the following consequence of rank-rigidity
\cite{caprace2011rank}.

\begin{lem}
\label{lem:rigid} Let $G$ be a
cocompactly cubulated group
with centre containing $Z\cong\Z$.
Then $G$ has a finite index subgroup $G_0\times Z$ with $G_0$ cocompactly
cubulated.
\end{lem}

\begin{proof}
Suppose that $G$ acts properly and cocompactly by cubical
automorphisms on a CAT(0) cube complex $X$. By \cite[Corollary~6.4(iii)]{caprace2011rank},
if we replace~$X$ with its essential core, and $G$ with a finite-index
subgroup, we obtain a cubical product
decomposition of $X$ respected by $G$, such that for each
factor there is an element of $g\in G$ acting on it as a rank one isometry.
Let $X_V$ be a factor on which $Z$ acts freely, and combine
all other factors into $X_H$, so that $X=X_H\times X_V$.

Note that the generator $z$ of $Z$ acts on $X_V$ as a rank one
isometry. Otherwise an axis of $g$ as above would not be
parallel to an axis of $z$. Hence $g$ and $z$ would generate $\Z^2$
acting properly on $X_V$, contradicting the fact that $g$ has rank one.
Consider $\mathrm{Min}(Z) = \R\times Y\subseteq X_V$, where $\R$ is an axis of $Z$.
Since $Z$ is contained in the centre of $G$, we have an induced
action of $G$ on $\R\times Y$ respecting this decomposition.
Since $z$ has rank one, we have that $Y$ does not contain a geodesic ray, and
hence is bounded. Consequently, $Y$ contains a fixed-point of
the action of~$G$. Thus $X_V$ contains a $G$--invariant line
$l$.

Let $\rho:G\to \mathrm{Isom}(l)$ be the induced map. Note that $\rho(G)$ does not flip the ends of $l$.
Moreover, since $X_V$ is a cube complex, the translation lengths on $l$ are discrete. Thus the image of $\rho$ can be identified with $\Z$,
which contains $\rho(Z)$ as a finite index subgroup.
Let $G_0=\mathrm{ker}(\rho)$. Thus $Z\times G_0$
is a finite index subgroup of $G$.
Moreover, $G_0$ acts properly
by cubical automorphisms on $X_H\subset X$. Since the action of
$Z$ on $X_V$ is proper, the action of $G_0$ on $X_H$ is
cocompact.
\end{proof}

We complement Lemma~\ref{lem:rigid} with the following:

\begin{lem}
\label{lem:qi}
Let $G=G_0\times Z$ be finitely generated, with $Z\cong\Z$. Let $H<G$ be a finite product of finitely generated free
groups of rank $\geq 2$ that is quasi-isometrically embedded.
\begin{enumerate}[(i)]
\item
The map $H\to G/Z$ is a quasi-isometric embedding.
\item
Let $G$ be cocompactly cubulated. If we require that $H\cap Z$ is trivial, then assertion (i) holds
also if in the product we allow free groups of rank~$1$.
\end{enumerate}
\end{lem}

\begin{proof}
If $H$ is a free group of rank $\geq 2$, then we choose in $H$ a free generating set~$S^\pm$.
In~$Z$ we consider the generating set $\{\pm1\}$ and in $G_0$ any symmetric generating set.
Let $|\cdot|_H,|\cdot|_{Z},|\cdot|_{G_0}$ denote
the corresponding word-lengths. Let $\pi_{G_0},\pi_Z$ be
the coordinate projections from $G$ to $G_0,Z$, respectively.
By assumption, there exists a constant~$c$ such that for any
$h\in H$, we have $|h|_H\leq
c\big(|\pi_{G_0}(h)|_{G_0}+|\pi_{Z}(h)|_{Z}\big)$.
Viewing~$h$ as a reduced word over $S^\pm$, choose $s\in S^\pm$ such that the word
$w=hsh^{-1}s^{-1}$ is reduced. Then $|\pi_{Z}(w)|_{Z}=0$, and
applying the above inequality with~$w$ in place of~$h$ we
obtain
$2|h|_H+2\leq
c|\pi_{G_0}(w)|_{G_0}\leq 2c\big(|\pi_{G_0}(h)|_{G_0}+|\pi_{G_0}(s)|_{G_0}\big)$. Consequently
$|h|_H\leq c|\pi_{G_0}(h)|_{G_0} +a$ for some uniform constant
$a$, and thus the restriction of $\pi_{G_0}$ to $H$ is a
quasi-isometric embedding, as desired.

Similarly, if $H$ is a product of free groups $H_i$ of rank $\geq 2$, then we choose generating
sets $S^\pm_i$ in $H_i$. Let $h=\prod h_i$ with $h_i\in H_i$. To get
an estimate on $|h|_H$, it suffices to use a product of reduced words $w=\prod
h_is_ih_i^{-1}s_i^{-1}$, with $s_i\in S_i^\pm$. This proves assertion (i).

If $G$ is cocompactly cubulated, then by Lemma~\ref{lem:rigid}, after passing to a finite index subgroup, the quotient $G/Z$
acts properly and cocompactly on a CAT(0)
cube complex $X$. Let $H=\Z^n\times H_0\leq G$, where $H_0$ is a finite product of finitely generated free
groups of rank $\geq 2$. We keep the notation $H$ for the isomorphic image of $H$ in $G/Z$. Then~$H$ preserves $\mathrm{Min}(\Z^n) = \R^n\times Y \subseteq X$ and respects its product structure.
We fix $v\in \R^n$ and $y\in Y$. From assertion~(i), the orbit map $h_0\to (h_0\cdot v,h_0\cdot y)$ from $H_0$
to $\R^n\times Y$ is a quasi-isometric embedding. Since the commutator
of $H_0$ acts trivially on the $\R^n$ factor, using the same
argument as for assertion (i), we obtain~$c$ satisfying $|h_0|_{H_0}\leq cd_Y(y,h_0\cdot y)$.
On the other hand, there is $c'$ such that for $f\in \Z^n$ we have $|f|_{\Z^n}\leq c'd_{\R^n}(v,f\cdot v)$.
Let $d$ be the maximum of the displacements $d_{\R^n}(v,s\cdot v)$ over the
generators $s$ of $H_0$. For $fh_0\in H$ consider the supremum
norm $\|fh_0\|=\sup \{|f|_{\Z^n}, 2c'd|h_0|_{H_0}\}$. If $|f|_{\Z^n}\geq
2c'd|h_0|_{H_0}$, then $$c'd_{\R^n}(v,fh_0\cdot v)\geq |f|_{\Z^n}-c'd|h_0|_{H_0}\geq
\frac{1}{2}|f|_{\Z^n}\geq \frac{1}{2}\|fh_0\|.$$ Otherwise, if
$|f|_{\Z^n}<
2c'd|h_0|_{H_0}$, then $$cd_Y(y, fh_0\cdot y)=cd_Y(y,h_0\cdot y)\geq
|h_0|_{H_0}>\frac{1}{2c'd}\|fh_0\|.$$ This proves assertion (ii).
\end{proof}

\begin{proof}[Proof of Theorem~\ref{thm:3gen}]

The implication (i)$\Rightarrow$(ii) is obvious. The implication
(iii)$\Rightarrow$(i) follows from Theorem~\ref{thm:constr}
unless the defining graph $\Gamma$ of $A$ has two edges $(a,c),(b,c)$ with label
$2$. By Theorem~\ref{thm:constr}, $A_{ab}$ is the fundamental
group of a nonpositively curved cube complex $K$. Then $K\times
S^1$ is a nonpositively curved cube complex with fundamental
group $A$.

The implication (ii)$\Rightarrow$(iii) follows from
Theorem~\ref{thm:2dim} if $A$ is $2$-dimensional. Suppose now
that $A$ is not $2$-dimensional. Then the labels of $\Gamma$
are $m_{ab}=3,m_{bc}=2$, and $m_{ac}=3,4,$ or $5$. Let $Z$ be
the centre of $A$ described in Lemma~\ref{lem:proj}(i).

Suppose that there exists a normal finite index subgroup $\hat
A\leq A$ that is cocompactly cubulated.
Let $\hat Z=\hat
A\cap Z$. By Lemma~\ref{lem:rigid}, up to replacing
$\hat A$ with a further finite index subgroup, we have
we have $\hat A=\hat A_0\times \hat Z$, where $\hat A_0$ is
cocompactly cubulated.
We keep the notation $\hat
A_0$ for its isomorphic image in the quotient $A/Z$. Note that
$\hat A_0\leq A/Z$ is a normal finite index subgroup.

By Theorem~\ref{K(pi,1)}(B), the Artin group $A$ is the
fundamental group of a $3$-dimensional cell complex which is
a~$K(A,1)$. Thus, by Lemma~\ref{rank}, the asymptotic rank of
$\hat A$ is~$\leq 3$. Hence the
asymptotic rank of $\hat A_0$ is $\leq 2$.

By Lemma~\ref{lem:proj}(ii), the intersections of $A_{ab}$ and
$A_{bc}$ with $Z$ are trivial. Thus $A_{ab}$ and $A_{bc}$ embed
into $A/Z$ under the quotient map, and we keep the notation
$A_{ab}$ and $A_{bc}$ for their images in $A/Z$.
By Lemma~\ref{lem:proj}(iii) in $A/Z$ we have $A_{ab}\cap
A_{bc}=A_b$.

Let $\hat A_{ab}=F_{k}\times \Z$
be a finite index subgroup of $A'_{ab}\cap \hat A_0$ guaranteed by Lemma~\ref{lem:free
times z}(1). We can assume that $\hat A_{ab}$ is normal in $A'_{ab}$. Let $\hat A_{bc}=\ A_{bc}\cap
\hat A_0=\Z^2$. By Lemmas~\ref{lem:special convex} and~\ref{lem:qi}(ii),
$\hat A_{ab},\hat A_{bc}<A/Z$ are quasi-isometric embeddings.

From this point we argue to reach a contradiction
exactly as in part (1) of the proof of Theorem~\ref{thm:2dim}.
\end{proof}

\bibliographystyle{alpha}
\bibliography{2-dim-artingroups10}

\end{document}